\documentclass[12pt]{amsart}

\usepackage{bm,latexsym,amsfonts,amsmath,amssymb,amsthm,url,graphics,psfrag,amscd,stmaryrd, mathrsfs}
\usepackage[english]{babel}
\usepackage[T1]{fontenc}

\usepackage{graphics}
\usepackage{graphicx}
\usepackage{color}

\newcommand{\R}{\mathbb{R}}

\newcommand{\N}{\mathbb{N}}
\newcommand{\E}{\mathbb{E}}

\newcommand{\pp}{\mathbb{P}}

\newcommand{\kO}{O}

\newcommand{\lin}{\left[\kern-0.15em\left[}
\newcommand{\rin} {\right]\kern-0.15em\right]}
\newcommand{\linf}{[\kern-0.15em [}
\newcommand{\rinf} {]\kern-0.15em ]}
\newcommand{\ilin}{\left]\kern-0.15em\left]}
\newcommand{\irin} {\right[\kern-0.15em\right[}

\newtheorem{lem}{Lemma}[section]

\newtheorem{theo}[lem]{Theorem}

\newtheorem{cor}[lem]{Corollary}

\newtheorem*{ack}{Acknowledgments}

\title[  Cram\'{e}r moderate deviation for critical Curie-Weiss model]
       {\bf A Cram\'{e}r type  moderate deviation theorem  for the critical Curie-Weiss model}

\author{Van Hao Can}
\author{Viet-Hung Pham}
\address{Institute of Mathematics, Vietnam Academy of Science and Technology, 18 Hoang Quoc Viet Street, 10307 Hanoi, Vietnam}
\email{cvhao89@gmail.com}
 \email{pgviethung@gmail.com}


\begin{document}
\maketitle
\begin{abstract}
Limit theorems for the magnetization of Curie-Weiss model have been studied extensively by Ellis and Newman. To refine these results, Chen, Fang and  Shao prove Cram\'{e}r type moderate deviation theorems for non-critical cases by using Stein method. In this paper, we consider the same question for the remaining case - the critical Curie-Weiss model. By direct and simple arguments based on Laplace method, we provide an explicit formula of the error and deduce a Cram\'{e}r type result.  
\end{abstract}

\section{Introduction}
  Let $(X_i)$ be a sequence of  i.i.d. random variables satisfying $\E X_1 =0$, $\textrm{Var}(X_1) =1$. Then the classic Central limit theorem says that the normalized sum $W_n = (X_1 + \ldots + X_n)/ \sqrt{n}$ converges in law to a standard normal random variable  $W$. A natural question is to understand the rate of the convergence of the tail probability  $\pp(W_n >x)$ to $\pp(W>x)$ for the largest possible  range of $x$. There are two major approaches to measure the approximation error. The first approach is to study the absolute error by Berry-Esseen type bounds. The other one is to study the relative error of  the tail probability. One of the first result in this approach is the following Cram\'er type moderate deviation theorem. If   $\E(e^{\alpha |X_1|^{1/2}}) < \infty$, for some $\alpha >0$, then   
\begin{eqnarray*}
\frac{\pp( W_n >x)}{1-\Phi(x)} = 1+ \kO(1)(1+x^3)/\sqrt{n},
\end{eqnarray*}
for $0 \leq x \leq n^{1/6}$, with $\Phi$ the standard normal distribution function. It has been also shown that the assumptions on the exponential moment of $X_1$ and the length of range $n^{1/6}$ are optimal.   We refer the reader to the book \cite{P} for a proof of this result and a more detailed discussion. 

\vspace{0.2 cm}
The Cram\'{e}r type moderate deviation  results  have been proved to be useful in  designing  statistical tests since they give  a relation between the size and the accuracy of  tests, see e.g. \cite{LiSh1,LiSh}. Hence, a lot of attention has been drawn in investigating this  problem  not only for independent variables but also  for  dependent structures as stationary process \cite{COT,WuZh}, self-linear process \cite{PSZW}, normalized sums \cite{CSWX,Shao},  and  L-statistics \cite{Grib}. On the other hand, Cram\'{e}r type moderate deviation theorems for nonnormal limit distribution are also provided, such as for chi-squared distribution \cite{LiSh}, for sub-Gaussian or exponential distribution \cite{CS}. 

\vspace{0.2 cm}
In this paper, we  study the case of   the {\it critical Curie-Weiss model}, where the spin variables are dependent and the limit distribution is nonnormal. Let us first recall some definitions and existing results for Curie-Weiss model.  For $n\in \N$, let $\Omega_n = \{\pm 1\}^n$  be the space of spin configurations. The spin configuration probability is given by Boltzman-Gibbs distribution, i.e. for any $\sigma \in \Omega_n$, 
\begin{eqnarray*}
\mu_n(\sigma)=Z_{n}^{-1} \exp \left( \frac{\beta}{n} \sum \limits_{1\leq i< j \leq n}  \sigma_i \sigma_j + \beta h \sum \limits_{i=1}^n \sigma_i\right),
\end{eqnarray*}
where $Z_n$ is the normalizing factor, $\beta >0$ and $h \in \R$ are inverse temperature and external field respectively. The Curie-Weiss model has been shown to exhibit a phase transition at $\beta_c=1$. More precisely, the  asymptotic behavior of the total spin (also called the magnetization) $S_n = \sigma_1 + \ldots + \sigma_n$ changes when $\beta$ crosses the critical value $1$. Let us consider the following fixed-point equation 
\begin{equation} \label{fpe}
m=\tanh(\beta(m+h)).
\end{equation}

\vspace{0.2 cm}
\noindent {\bf Case 1.} $ 0 <\beta <1$, $h \in \R$ or $\beta \geq 1, h \neq 0$ ({\it the uniqueness regime of magnetization}). The equation \eqref{fpe} has   a unique solution $m_0$, such that $m_0h \geq 0$.  In this case, $S/n$ is concentrated around $m_0$ and has a
Gaussian limit under proper standardization, see \cite{E}. Moreover,  in \cite{CFS} the authors prove  the following moderate deviation theorem for the magnetization by using  Stein method. 
\begin{theo} \label{tm} \cite[Proposition 4.3]{CFS} In case 1, let us define
$$W_n= \frac{S_n - nm_0}{v_n},$$
where
$$v_n = \sqrt{\frac{n(1-m_0^2)}{1-(1-m_0^2)\beta}}.$$
Then we have
\begin{eqnarray*}
\frac{\mu_n(\sigma: W_n >x)}{1-\Phi(x)} = 1+ \kO(1)(1+x^3)/\sqrt{n},
\end{eqnarray*}
for $0 \leq x \leq n^{1/6}$.
\end{theo}

\vspace{0.2 cm}
\noindent
{\bf Case 2.} $\beta >1, h = 0$ ({\it the low temperature regime without external field}). The equation \eqref{fpe} has  two nonzero solutions $m_1 < 0 < m_2$,
where $m_1 = - m_2$. In this case, one has the conditional central limit theorems as follows: conditionally on $S_n < 0$ (resp. $S_n > 0$), $S/n$ is concentrated around $m_1$ (resp. $m_2$) and  has a Gaussian limit after  proper scaling, see \cite{E}.  Similarly to case 1, a moderate deviation  result has been also proved.
\begin{theo} \label{th} \cite[Proposition 4.4]{CFS} In case 2, let us define
\begin{eqnarray*}
W_{1,n}=\frac{S_n -n m_1}{v_{1,n}} \hspace{1 cm} \textrm{and} \hspace{1cm} W_{2,n}=\frac{S_n -n m_2}{v_{2,n}},
\end{eqnarray*}
where
$$v_{1,n} = \sqrt{\frac{n(1-m_1^2)}{1-(1-m_1^2)\beta}} \hspace{1 cm} \textrm{and} \hspace{1cm} v_{2,n} = \sqrt{\frac{n(1-m_2^2)}{1-(1-m_2^2)\beta}}.$$
Then we have
\begin{eqnarray*}
\frac{\mu_{n}(\sigma: W_{1,n}>x \mid S_n<0)}{1- \Phi(x)} = 1 + \kO(1)(1+x^3)/\sqrt{n},
\end{eqnarray*}
and 
\begin{eqnarray*}
\frac{\mu_{n}(\sigma:W_{2,n}>x \mid S_n>0)}{1- \Phi(x)} = 1 + \kO(1)(1+x^3)/\sqrt{n},
\end{eqnarray*}
for $0 \leq x \leq n^{1/6}$.
\end{theo} 

\vspace{0.2 cm}
\noindent
{\bf Case 3.} $\beta =1$ and $h=0$ ({\it the critical case}). The equation \eqref{fpe} has a unique solution $0$ and $S/n$ is concentrated around $0$. In this case, $S_n/ n^{3/4}$  converges to a nonnormal  distribution with density proportional to $e^{-x^4/12}$, see \cite{E, EN}. Moreover,   the authors of \cite{CS, EL}  give  Berry-Esseen type  bounds for this convergence. 

\begin{theo} \label{tcs} \cite[Theorem 2.1] {CS}    In case 3, let us  define 
$$W_n=\frac{S_n}{n^{3/4}}.$$
Then there exists a positive constant $C$, such that for all $x$
\begin{equation}
\limsup \limits_{n\rightarrow \infty} \sqrt{n} \, \Big| \, \mu_n \left(\sigma : W_n \leq x \right) - F(x)\, \Big | \, \leq C,
\end{equation}
where  
$$F(x) = \frac{\int_{-\infty}^{x}e^{-t^4/12} dt}{\int_{-\infty}^{\infty}e^{-t^4/12}dt}.$$
\end{theo}
We remark that  in \cite{EL}, the authors generalize Theorem \ref{tcs}  to a near critical regime of inverse temperature $\beta = 1 + O(\tfrac{1}{\sqrt{n}})$. They  also consider  a general class of Curie-Weiss model, where the distribution of a single spin is a generic probability measure instead of Bernoulli distribution as in the classical  model.    

\vspace{0.2 cm}
In this paper, we will prove a Cram\'er type moderate deviation theorem for the total spin  in the critical case.  Our main result is as follows.
\begin{theo} \label{tcp} For the critical case, when $\beta =1$ and $h=0$, let us  define 
$$W_n=\frac{S_n}{n^{3/4}}.$$
Then there exists a positive constant $C$, such that for all $n$ large enough and $0 \leq x \leq n^{1/12}$,
\begin{eqnarray}
\Big |\frac{\mu_n(\sigma: W_n >x)}{1-F(x)} - 1 - \frac{G(x)}{\sqrt{n}} \Big | \leq \frac{C(x^{12}+n^{1/3})}{n},
\end{eqnarray}
where 
$$F(x)= \frac{\int_{\infty}^x p_1(t)dt}{\int_{-\infty}^{\infty} p_1(t)dt},$$
and
\begin{eqnarray*}
G(x) &=&  \left( \frac{\int^{\infty}_x p_2(t)dt}{\int_{x}^{\infty} p_1(t)dt} - \frac{\int^{\infty}_{-\infty} p_2(t)dt}{\int_{-\infty}^{\infty} p_1(t)dt}  \right)
\end{eqnarray*}
with 
\begin{eqnarray*}
p_1(t) = e^{-\frac{t^4}{12}} \hspace{1cm} \textrm{and} \hspace{1cm} p_2(t) = \left(\frac{t^2}{2} - \frac{t^6}{30} \right) e^{-\frac{t^4}{12}}.
\end{eqnarray*}
\end{theo}

  It is worth noting that Theorem \ref{tcp} gives the exact formula of the error term of order $n^{-1/2}$, while  moderate deviation results  in Theorems \ref{tm} and \ref{th} only show asymptotic estimates of the error terms. The range of  estimate $n^{1/6}$ is replaced by $n^{1/12  }$ due to the change of scaling and limit distribution. The  proof of Theorem \ref{tcp} is simple and direct, based on Laplace method-like arguments.       We have a direct corollary.
\begin{cor} \label{cor} For $0 \leq x \leq n^{1/12}$, we have
\begin{eqnarray*}
\frac{\mu_n(\sigma: W_n >x)}{1-F(x)} = 1+ \kO(1)(1+x^6)/\sqrt{n}.
\end{eqnarray*}
Moreover, for any fixed real number $x$, 
\begin{equation*}
\lim\limits_{n\rightarrow \infty} \sqrt{n} \, \Big( \, \mu_n \left(\sigma : W_n \leq x \right) - F(x)\, \Big) = (F(x)-1)G(x). 
\end{equation*}
\end{cor}
The first part of this corollary is a Cram\'{e}r moderate deviation result in classic form, whereas the second part is an improvement of Theorem \ref{tcs}.

\vspace{0.2 cm}
The paper is organized as follows. In Section 2, we provide some preliminary results. In Section 3, we prove the main theorem \ref{tcp}.

\vspace{0.2 cm}
 We fix here some notation. If $f $ and $g$ are two real functions, we write $f= O(g)$ if there exists a constant $C>0,$ such that $f(x) \leq C g(x)$ for all $x$; $f= \Omega(g)$ if $g= \kO(f)$; and  $f =\Theta(g) $ if $f= O(g)$ and $g= O(f)$.  
  \section{Preliminaries}
  \subsection{A lemma on the integral approximations}
  \begin{lem} \label{loia}
   Let $m,q,p$ be positive real numbers. 
  \begin{itemize}
  \item[(i)]   Assume that  $f(t)$ is a decreasing function in $[(m-1)/p, (q+1)/p]$. Then 
  \begin{eqnarray*}
 \Big| \sum_{\substack{ m <\ell < n  \\ 2 \mid \ell  }} f \left( \frac{\ell}{p} \right) - \frac{p}{2}\int_{m/p}^{q/p} f(t) dt \Big| \leq \Big| f \left( \frac{m}{p} \right) \Big|  + \Big| f \left( \frac{q}{p} \right)\Big| , 
  \end{eqnarray*}
  and 
  \begin{eqnarray*}
 \Big| \sum_{\substack{ m <\ell < q  \\ 2 \nmid \ell  }} f \left( \frac{\ell}{p} \right) - \frac{p}{2}\int_{m/p}^{q/p} f(t) dt \Big| \leq  \Big| f \left( \frac{m}{p} \right)\Big|   + \Big| f \left( \frac{q}{p} \right) \Big| . 
  \end{eqnarray*}
  \item[(ii)]  Assume that $f(t)$ is a differentiable function on $\R$ and there exists a positive constant $K$, such that $|f(t)| +|f'(t)| \leq K$.  Then 
   \begin{eqnarray*}
 \Big | \sum_{\substack{ m <\ell < q  \\ 2 \mid \ell  }} f \left( \frac{\ell}{p} \right) - \frac{p}{2}\int_{m/p}^{q/p} f(t) dt \Big | \leq  \frac{K(q-m)}{p} + 2K, 
  \end{eqnarray*}
  and 
  \begin{eqnarray*}
 \Big | \sum_{\substack{ m <\ell < q  \\ 2 \nmid \ell  }} f \left( \frac{\ell}{p} \right) - \frac{p}{2}\int_{m/p}^{q/p} f(t) dt \Big | \leq \frac{K(q-m)}{p} + 2K. 
  \end{eqnarray*}
  \end{itemize}
  \end{lem}
  \begin{proof}
The proof of  (i) is simple, so we safely leave it to the reader. For (ii),   by using the mean value theorem, we get that for any $\ell$, 
\begin{eqnarray*}
\Big| \, f \left(\frac{\ell}{p} \right) - \frac{p}{2} \int_{\frac{\ell}{p }}^{\frac{\ell+2}{p }} f(t)dt  \, \Big | &\leq &  \frac{pK}{2} \int_{\frac{\ell}{p }}^{\frac{\ell+2}{p }} \left( t- \frac{\ell}{p} \right) dt =\frac{K}{p}.
\end{eqnarray*}
Therefore, by summing over $\ell$ we get desired results. 
  \end{proof}
  \subsection{Estimates on the binomial coefficients} We first recall  a version of Stirling approximation (see \cite{R}) that for all $n\geq 1$, 
\begin{equation*}
\log (\sqrt{2 \pi n}) + n \log n - n + \frac{1}{12n +1} \leq \log (n!) \leq \log (\sqrt{2 \pi n}) + n \log n - n + \frac{1}{12n}.
\end{equation*} 
Using this approximation, we can show that 
\begin{eqnarray}
\binom{n}{k}  \leq   e^{n I(k/n)}, \hspace{0.6 cm} \textrm{for all } k = 0, \ldots, n,  \label{ubnh}
\end{eqnarray}
and
\begin{eqnarray}
\binom{n}{k} = (1-\kO(n^{-1})) \sqrt{\frac{n}{2 \pi k(n-k)}}  \times e^{n I(k/n)}, \hspace{0.6 cm} \textrm{for } |k - (n/2)| < n/4, \label{bnh}
\end{eqnarray}
where $I(0)=I(1)=0$ and for $t \in (0,1)$,
$$I(t)=(t-1) \log (1-t)-t \log t.$$
 We will see in Section 3.1  that the function $J(t)$ defined by 
\begin{equation} \label{doj}
J(t)=I(t)+ \frac{(2t-1)^2}{2}
\end{equation} 
plays an important role in the expression of the distribution  function of $W_n$.  We prove here a lemma to describe  the behavior of $J(t)$. 
\begin{lem} \label{ljx} Let $J(t)$ be the function defined as in \eqref{doj}. Then 
\begin{itemize}
\item[(i)] $J'(1/2)=J''(1/2)=J'''(1/2)=J^{(5)}(1/2)=J^{(7)}(1/2)=0$, and for all $t \neq 1/2$ 
$$J''(t)<0.$$
\item[(ii)] $J^{(4)}(1/2)=-32$, $J^{(6)}(1/2)=-1536$, and for all $1/4 \leq t \leq 3/4$
$$  -2^{25}< J^{(8)}(t) <0. $$ 
\end{itemize}
\end{lem}
\begin{proof}
We have 
$$J'(t)= \log \left( \frac{1-t}{t} \right)+ (4t-2).$$
Hence 
\begin{eqnarray*}
J''(t)= - \big[ t^{-1}+(1-t)^{-1} \big] +4, \hspace{0.4 cm} J'''(t)= \big[ t^{-2}-(1-t)^{-2} \big], \hspace{0.4 cm} J^{(4)}(t)= - 2\big[ t^{-3}+(1-t)^{-3} \big], \\
J^{(5)}(t)= 6\big[ t^{-4}-(1-t)^{-4} \big], \hspace{0.4 cm} J^{(6)}(t)= -24\big[ t^{-5}+(1-t)^{-5} \big], \hspace{0.4 cm} J^{(7)}(t)= 120\big[ t^{-6}-(1-t)^{-6} \big], 
\end{eqnarray*}
and 
$$J^{(8)}(t)= -720\big[ t^{-7}+(1-t)^{-7} \big].$$
Using these equations, we can deduce the  desired results. 
\end{proof}

\section{Proof of Theorem \ref{tcp}}  
 \subsection{An expression of the distribution function of $W_n$}
 Let us denote by $F_n(x)$  the distribution function of $W_n$, i.e.  for $x \in \R$
 \begin{eqnarray} \label{fnx}
 F_{n}(x)=\mu_n ( \sigma: W_n \leq x) = \mu_n(\sigma: \sigma_1 + \ldots + \sigma_n \leq  n^{3/4} x). 
 \end{eqnarray}
For $\sigma \in \Omega_n$, we define
  $$\sigma_+=\{i: \sigma_i=1\}.$$
Observe that if $|\sigma_{+}|=k$, then 
\begin{eqnarray*} 
\frac{1}{n}\sum \limits_{i\leq j} \sigma_i \sigma_j = 1 +\frac{1}{n}\sum \limits_{i< j} \sigma_i \sigma_j &=& 1 + \frac{1}{2n} \left( \left(\sum \limits_{1\leq i \leq n} \sigma_i \right)^2 -n \right) \notag \\
&=& \frac{(2k-n)^2}{2n}+ \frac{1}{2}.
\end{eqnarray*}
Hence, 
\begin{eqnarray*} 
Z_n= \sum \limits_{\sigma \in \Omega_n} \exp\left(\frac{1}{n}\sum \limits_{i\leq j} \sigma_i \sigma_j \right) =\sum_{k=0}^n \sum_{\substack{\sigma \in \Omega_n \\ |\sigma_+| =k}} \exp\left(\frac{1}{n}\sum \limits_{i\leq j} \sigma_i \sigma_j \right) =  \sum_{k=0}^n \binom{n}{k} e^{\frac{(2k-n)^2}{2n}+ \frac{1}{2}}. 
\end{eqnarray*}
Let us define
$$x_{k,n}=\binom{n}{k} e^{\frac{(2k-n)^2}{2n}+ \frac{1}{2}}.$$
Then 
\begin{eqnarray} \label{zn}
Z_n =  \sum_{k=0}^n x_{k,n},
\end{eqnarray}
and
\begin{equation} \label{spk}
\mu_n(\sigma: |\sigma_+| = k) = \frac{x_{k,n}}{Z_n}.
\end{equation}
Combining \eqref{fnx}, \eqref{zn} and \eqref{spk} yields that 
\begin{eqnarray} \label{mfn}
1- F_n(x) &=&  \mu_n(\sigma: \sigma_1 + \ldots + \sigma_n >  n^{3/4} x) \notag \\
&=& \mu_n(\sigma: 2|\sigma_+| -n > n^{3/4} x) = \mu_n \left(\sigma: |\sigma_+| > \frac{n + n^{3/4} x}{2} \right) \notag \\
&=& \frac{1}{Z_n} \sum \limits_{k =0}^n x_{k,n} \mathbb{I} \left(k > \frac{n + n^{3/4} x}{2}\right), 
\end{eqnarray}
where $\mathbb{I}(\cdot)$ stands for the indicator function. 
Using \eqref{ubnh} and \eqref{bnh}, we obtain 
\begin{eqnarray}
x_{k,n}&  \leq &   e^{n J(k/n) + 1/2}, \hspace{0.6 cm} \textrm{for all } k = 0, \ldots, n,  \label{uxkn}
\end{eqnarray}
and
\begin{eqnarray}
x_{k,n}& = & (1-\kO(n^{-1})) \sqrt{\frac{n}{2 \pi k(n-k)}} \times e^{n J(k/n) + 1/2},  \hspace{0.6 cm} \textrm{for } |k - (n/2)| < n/4, \label{xkn}
\end{eqnarray}
with  $J(t)$ the function defined  in \eqref{doj}. 

By Lemma \ref{ljx}, we observe  that $J(t)$ attains the maximum at the unique point $\tfrac{1}{2}$. This fact suggests us  that the value of $Z_n$ (the sum of $(x_{k,n})$) is concentrated at the middle terms. Let us define 
$$y_n= \sqrt{\frac{2}{ \pi n}} \times e^{n J(1/2) + 1/2},$$
which is asymptotic to $x_{[n/2],n}$. We define also
$$y_{k,n}=\frac{x_{k,n}}{y_n}.$$
Then the equation \eqref{mfn} becomes
\begin{eqnarray} \label{mtfn}
1 - F_n(x)=\frac{1}{\sum \limits_{k =0}^n  y_{k,n}} \times \sum \limits_{k =0}^n  y_{k,n}  \mathbb{I}\left(k > \frac{n + n^{3/4} x}{2} \right).
\end{eqnarray} 
Moreover, using \eqref{uxkn} and \eqref{xkn}, we  obtain estimates on $(y_{k,n})$, 
\begin{eqnarray}
y_{k,n}& \leq & \sqrt{\frac{\pi n}{2}} \times e^{n \big[J(k/n) -J( 1/2) \big]}, \hspace{0.6 cm} \textrm{for all } k = 0, \ldots, n,   \label{uxly}
\end{eqnarray} 
and 
\begin{eqnarray}
y_{k,n}& = & (1-\kO(n^{-1}))\sqrt{\frac{n^2}{4  k(n-k)}} \times e^{n \big[J(k/n) -J( 1/2) \big]}  \hspace{0.6 cm} \textrm{for } |k - \tfrac{n}{2}| < \tfrac{n}{4}   \label{xly}.
\end{eqnarray}
We define 
\begin{eqnarray*}
A_n &=& \sum \limits_{k=0}^n y_{k,n} \mathbb{I}\left( \big | k - \frac{n}{2} \big| \geq \frac{n}{4}  \right), \hspace{0.2 cm} B_n = \sum \limits_{k=0}^n y_{k,n} \mathbb{I}\left( \big | k - \frac{n}{2} \big| <\frac{n}{4}  \right) \\
\hat{A}_n &=& \sum \limits_{k=0}^n y_{k,n} \mathbb{I}\left(  k - \frac{n}{2}  \geq \frac{n}{4}  \right), \hspace{0.2 cm} B_{n,x} = \sum \limits_{k=0}^n y_{k,n} \mathbb{I} \left( \frac{n}{4}  > k - \frac{n}{2}  > \frac{n^{3/4}x}{2}  \right).
\end{eqnarray*}
Then by \eqref{mtfn},
\begin{equation} \label{fab}
1-F_{n}(x)=\frac{\hat{A}_n+B_{n,x}}{A_n+B_n}.
\end{equation}

 \subsection{Estimates of $A_n$ and $\hat{A}_n$}
\begin{lem} \label{eaa}
There exists a positive constant $c$, such that for $n$ large enough,
$$\hat{A}_n \leq A_n \leq e^{-cn}.$$
\end{lem}
\begin{proof}
By Lemma \ref{ljx}, we have $J'(\tfrac{1}{2})=0$ and $J''(t) \leq 0$ for all $t\in (0,1)$. Therefore, 
\[\max_{|x-0.5|\geq 0.25} J(x) = \max \{J(0.75), J(0.25)\} =J(0.25).\]
 Hence  for all $|k-(n/2)|\geq n/4$,
 \begin{equation*} \label{sach}
 J(k/n)-J(1/2) \leq J(0.25) - J(0.5) < -0.005.
\end{equation*}  
Thus for all $|k-(n/2)|\geq n/4$,
\begin{equation} \label{nlj}
n \big(J(k/n)-J(1/2)\big) <- 0.005 n. 
\end{equation}
It follows from  \eqref{uxly} and \eqref{nlj} that  for   $|k-(n/2)|\geq n/4$, 
\begin{equation*}
y_{k,n}\leq  \sqrt{2n} \exp \left(-0.005 n  \right).
\end{equation*}
 Thus 
 $$\hat{A}_n \leq A_n \leq  n \sqrt{2n} \exp \left(-0.005 n  \right) < \exp \left(-0.004 n  \right),$$
 for all $n$ large enough.
\end{proof}

 \subsection{Estimates of $B_n$} 
 
 By using Lemma \ref{ljx} (i) and Taylor expansion, we get   
 \begin{eqnarray*} \label{jkn}
J \left( \frac{k}{n} \right) - J \left( \frac{1}{2} \right) = \frac{1}{4!} J^{(4)}(1/2) \left( \frac{k}{n}-\frac{1}{2} \right)^4+\frac{1}{6!} J^{(6)}(1/2) \left( \frac{k}{n}-\frac{1}{2} \right)^6  + \frac{1}{8!} J^{(8)}(\xi_{k,n}) \left( \frac{k}{n}-\frac{1}{2} \right)^8,
 \end{eqnarray*}
 with some $\xi_{k,n}$ between $k/n$ and $1/2$. Hence, by Lemma \ref{ljx} (ii), 
 \begin{eqnarray*} \label{jkn}
J \left( \frac{k}{n} \right) - J \left( \frac{1}{2} \right) &\leq &   -\frac{(2k-n)^4}{12 n^4}-  \frac{(2k-n)^6}{30n^6}.
 \end{eqnarray*}
 and 
 \begin{eqnarray*} \label{jkn}
J \left( \frac{k}{n} \right) - J \left( \frac{1}{2} \right) &\geq &  -\frac{(2k-n)^4}{12 n^4}-  \frac{(2k-n)^6}{30n^6}  - \frac{2^{17}(2k-n)^8}{ n^8 8!}. 
 \end{eqnarray*}
 Therefore,
 \begin{eqnarray*} \label{jkn}
n(J(k/n)-J(1/2)) &\leq &   -\frac{(2k-n)^4}{12 n^3}-  \frac{(2k-n)^6}{30n^5}
 \end{eqnarray*}
 and 
 \begin{eqnarray*} \label{jkn}
n(J(k/n)-J(1/2)) &\geq &  -\frac{(2k-n)^4}{12 n^3}-  \frac{(2k-n)^6}{30n^5}  - \frac{2^{17}(2k-n)^8}{n^7 8!}. 
 \end{eqnarray*}
Combining the last two estimates with  the inequality that $1-x \leq e^{-x} \leq 1-x + \tfrac{x^2}{2}$ for all $x\geq 0$, we get
\begin{eqnarray*}
e^{n \big[ J(k/n)-J(1/2)\big]} &\leq& \exp \left(\frac{-(2k-n)^4}{12 n^3} \right) \left(1 -  \frac{(2k-n)^6}{30 n^5} +  \frac{(2k-n)^{12}}{1800 n^{10}} \right),
\end{eqnarray*}
and 
\begin{eqnarray*}
e^{n \big[ J(k/n)-J(1/2)\big]} &\geq& \exp \left(\frac{-(2k-n)^4}{12 n^3} \right) \left(1 -  \frac{(2k-n)^6}{30 n^5} \right)  \left(1 -  \frac{2^{17}(2k-n)^8}{ n^7 8!} \right). 
\end{eqnarray*}
Therefore,
\begin{eqnarray} \label{jtj}
e^{n \big[ J(k/n)-J(1/2)\big]} =  \exp \left(\frac{-(2k-n)^4}{12 n^3} \right) \left(1 -  \frac{(2k-n)^6}{30 n^5} + \kO(1) X_{k,n}\right),
\end{eqnarray}
where 
$$X_{k,n}= \frac{(2k-n)^8}{ n^7} + \frac{(2k-n)^{12}}{n^{10}}+ \frac{(2k-n)^{14}}{n^{12}}.$$
On the other hand, for $|k-(n/2)|<n/4$,
\begin{eqnarray} \label{ckn}
\sqrt{\frac{n^2}{4  k(n-k)}} = 1 +\frac{(2k-n)^2}{2 n^2} + \kO \left(\frac{(2k-n)^{4}}{n^{4}} \right).
\end{eqnarray}
Combining  \eqref{xly}, \eqref{jtj} and \eqref{ckn}, we have  for $|k-(n/2)|<n/4$,
\begin{equation*}
y_{k,n}= \left(1+ \frac{(2k-n)^2}{2n^2} -\frac{(2k-n)^6}{30n^5} + \kO(1) R_{k,n} \right) \exp \left(-\frac{(2k-n)^4}{12n^3} \right),
\end{equation*}
where 
$$R_{k,n}= \frac{1}{n}+\frac{(2k-n)^{4}}{n^{4}}+ \frac{(2k-n)^8}{ n^7} + \frac{(2k-n)^{12}}{n^{10}}+ \frac{(2k-n)^{14}}{n^{12}}.$$
By letting $\ell = 2k -n$, we obtain
\begin{eqnarray}
B_n &=& \sum \limits_{n/4<k<3n/4} y_{k,n} \notag \\
& =& \sum_{\substack{|\ell| < n/2  \\ 2 \mid (\ell + n) }}  e^{-\frac{\ell^4}{12 n^3}}\left(1+ \frac{\ell^2}{2n^2} -\frac{\ell^6}{30n^5} + \kO \left( \frac{1}{n}+\frac{\ell^{4}}{n^{4}}+ \frac{\ell^8}{ n^7} + \frac{\ell^{12}}{n^{10}}+ \frac{\ell^{14}}{n^{12}} \right) \right) \notag \\
& =& \sum_{\substack{|\ell| < n/2  \\ 2 \mid (\ell + n) }}  e^{-\left(\frac{\ell}{n^{3/4}}\right)^4/12}\left[ 1+ \frac{1}{\sqrt{n}} \left( \frac{1}{2}\left(\frac{\ell}{n^{3/4}}\right)^2 - \frac{1}{30}\left(\frac{\ell}{n^{3/4}}\right)^6  \right) \right.  \notag \\ 
&& \hspace{1cm} + \left. \frac{\kO(1)}{n} \left( 1+ \left(\frac{\ell}{n^{3/4}}\right)^4 + \left(\frac{\ell}{n^{3/4}}\right)^8 + \left(\frac{\ell}{n^{3/4}}\right)^{12} + \frac{1}{\sqrt{n}} \left(\frac{\ell}{n^{3/4}}\right)^{14} \right) \right] \notag \\
&=&  \sum_{\substack{|\ell| < n/2  \\ 2 \mid (\ell + n) }} p_1 \left( \frac{\ell}{n^{3/4}} \right) + \frac{1}{\sqrt{n}} \sum_{\substack{|\ell| < n/2  \\ 2 \mid (\ell + n) }} p_2 \left( \frac{\ell}{n^{3/4}} \right) + \frac{\kO(1)}{n}  \sum_{\substack{|\ell| < n/2  \\ 2 \mid (\ell + n) }} r \left( \frac{\ell}{n^{3/4}} \right), \label{eobn}
\end{eqnarray}
where 
\begin{eqnarray*}
p_1(t) &=& e^{-\frac{t^4}{12}} \\
p_2(t) & = & \left(\frac{t^2}{2} - \frac{t^6}{30} \right) e^{-\frac{t^4}{12}}\\
r(t) &=& (1 + t^{4} + t^8 + t^{12} + t^{14}/\sqrt{n} )e^{-\frac{t^4}{12}}.
\end{eqnarray*}
The proof of the following lemma is simple, so we omit it. 
\begin{lem} \label{lpb}
There exists a positive constant $K$, such that 
\begin{equation*} 
\sup \limits_{t \in \R} |p_1(t)|+|p_2(t)| +|r(t)| + |p_1'(t)|+|p_2'(t)| +|r'(t)| \leq K.
\end{equation*}
\end{lem}
 Using Lemma \ref{loia} (ii) and Lemma \ref{lpb}, we obtain that 
\begin{equation*}
\sum_{\substack{|\ell| \leq n^{5/6}    \\ 2 \mid (\ell + n) }} p_1 \left( \frac{\ell}{n^{3/4}} \right) = \frac{n^{3/4}}{2} \int \limits_{-n^{1/12}}^{n^{1/12}} p_1(t) dt +O(n^{1/12}).  
\end{equation*}
Moreover, 
\begin{eqnarray*}
\sum_{\substack{  n^{5/6} < |\ell| < n/2   \\ 2 \mid (\ell + n) }} p_1 \left( \frac{\ell}{n^{3/4}} \right) &\leq& n e^{-n^{1/3}}, \\
\int \limits_{|t| \geq n^{1/12}}  p_1(t) dt &=&o(n^{-1}).
\end{eqnarray*}
Combining the last three estimates gives that 
\begin{equation} \label{eop1}
\sum_{\substack{|\ell| < n/2    \\ 2 \mid (\ell + n) }} p_1 \left( \frac{\ell}{n^{3/4}} \right) = \frac{n^{3/4}}{2} \int \limits_{-\infty}^{ \infty } p_1(t) dt +O(n^{1/12}).  
\end{equation}
Similarly, 
\begin{eqnarray}
\sum_{\substack{|\ell| < n/2    \\ 2 \mid (\ell + n) }} p_2 \left( \frac{\ell}{n^{3/4}} \right) = \frac{n^{3/4}}{2} \int \limits_{-\infty}^{ \infty } p_2(t) dt +O(n^{1/12}), \label{eop2}\\
\sum_{\substack{|\ell| < n/2    \\ 2 \mid (\ell + n) }} r_1 \left( \frac{\ell}{n^{3/4}} \right) = \frac{n^{3/4}}{2} \int \limits_{-\infty}^{ \infty } r_1(t) dt +O(n^{1/12}).  \label{eor} 
\end{eqnarray}
We now can deduce  from \eqref{eobn}, \eqref{eop1}, \eqref{eop2} and \eqref{eor} an estimate of $B_n$ that 
\begin{eqnarray} \label{ebn}
&& B_{n} = \frac{n^{3/4}}{2} \int \limits_{-\infty}^{ \infty } p_1(t)   + \frac{n^{1/4}}{2}  \int \limits_{-\infty}^{ \infty } p_2(t) + O(n^{1/12}).
\end{eqnarray}
\subsection{Estimates of $B_{n,x}$}
Using the same arguments for \eqref{eobn}, we also have 
\begin{eqnarray} \label{eobnx}
B_{n,x} &=& \sum \limits y_{k,n} \mathbb{I} \left( \frac{n}{4}  > k - \frac{n}{2}  > \frac{n^{3/4}x}{2}  \right) \notag \\
&=&  \sum_{\substack{n^{3/4} x <\ell < n/2  \\ 2 \mid (\ell + n) }} p_1 \left( \frac{\ell}{n^{3/4}} \right) + \frac{1}{\sqrt{n}} \sum_{\substack{n^{3/4} x <\ell < n/2   \\ 2 \mid (\ell + n) }} p_2 \left( \frac{\ell}{n^{3/4}} \right) + \frac{\kO(1)}{n}  \sum_{\substack{n^{3/4} x <\ell < n/2   \\ 2 \mid (\ell + n) }} r \left( \frac{\ell}{n^{3/4}} \right).
\end{eqnarray}
In the sequel, we consider two cases: $x>10$ and $x\leq 10$. For the case $x>10$, we will use Part (i) of Lemma \ref{loia} to obtain a sharp estimate on $B_{n,x}$, while for the case $x\leq 10$, as for $B_n$, we apply Part (ii) to get a suitable estimate.  The choice of the number 10 is flexible. We just need the fact that the functions $p_1(\cdot),p_2(\cdot)$ and $r(\cdot)$  are decreasing in the interval $(c,\infty)$ for a positive constant $c$ (see Lemma \ref{lhd}).
\subsubsection{Case $x >10$} 
  \begin{lem}\label{lhd} These functions $p_1(t), p_2(t)$ and $r(t)$ are decreasing in $(9, \infty)$.
  \end{lem}
The proof of this lemma is elementary, so we omit it.    Applying  Lemma \ref{loia} (i) and Lemma \ref{lhd} to the sums in \eqref{eobnx}, we obtain 
  \begin{eqnarray*}
 && B_{n,x} = \frac{n^{3/4}}{2} \int\limits_x^{\tfrac{n^{1/4}}{2}} p_1(t) dt  + \frac{n^{1/4}}{2} \int\limits_x^{\tfrac{n^{1/4}}{2}} p_2(t) dt \\
 && + \kO(1) \left( \frac{1}{n^{1/4}} \int\limits_x^{\tfrac{n^{1/4}}{2}} r(t) dt + p_1(x) + \frac{p_2(x)}{\sqrt{n}} + \frac{r(x)}{n} + p_1\left(\tfrac{n^{1/4}}{2} \right) + \frac{ p_2\left(\tfrac{n^{1/4}}{2}\right)}{\sqrt{n}} + \frac{r\left(\tfrac{n^{1/4}}{2}\right)}{\sqrt{n}} \right).
  \end{eqnarray*}
Moreover,  
\begin{eqnarray*}
p_1\left(\tfrac{n^{1/4}}{2}\right),  p_2\left(\tfrac{n^{1/4}}{2}\right),  r\left(\tfrac{n^{1/4}}{2}\right), \int\limits_{\tfrac{n^{1/4}}{2}}^{\infty} p_1(t) dt, \int\limits_{\tfrac{n^{1/4}}{2}}^{\infty} p_2(t) dt, \int\limits_{\tfrac{n^{1/4}}{2}}^{\infty}  r(t) dt   = o(n^{-1}).
\end{eqnarray*}
Therefore,
\begin{eqnarray} \label{ebnx}
&& B_{n,x} = \frac{n^{3/4}}{2} \hat{P}_1(x)  + \frac{n^{1/4}}{2} \hat{P}_2(x)   + \kO(1) \left( \frac{\hat{R}(x)}{n^{1/4}}  + p_1(x) + \frac{p_2(x)}{\sqrt{n}} + \frac{r(x)}{n}  \right) + o(1),\end{eqnarray}
where 
\begin{eqnarray}
\hat{P}_1(x) &=&\int_x^{\infty} p_1(t) dt, \notag \\
\hat{P}_2(x) &=&\int_x^{\infty} p_2(t) dt, \label{p12} \\
\hat{R}(x) &=&\int_x^{\infty} r(t) dt. \notag
\end{eqnarray}

\subsubsection{Case $x \leq 10$} Using the same arguments for \eqref{ebn}, we can show that 
\begin{eqnarray} \label{fbnx}
 B_{n,x} = \frac{n^{3/4}}{2} \hat{P}_1(x)  + \frac{n^{1/4}}{2} \hat{P}_2(x)   + \kO(n^{1/12}).
\end{eqnarray}

\subsection{Conclusion}
We first rewrite \eqref{ebn} as 
\begin{eqnarray} \label{ebnn}
&& B_{n} = \frac{n^{3/4}}{2} \hat{P}_1(-\infty)    + \frac{n^{1/4}}{2} \hat{P}_2(-\infty)  + O(n^{1/12}), 
\end{eqnarray}
with $\hat{P}_1(x)$ and $\hat{P}_2(x)$ as in \eqref{p12}.
\subsubsection{Case $x>10$}

Combining  \eqref{fab},  \eqref{ebnx}, \eqref{ebnn},  we have 
\begin{eqnarray*}
1-F_n(x) &=& \frac{n^{3/4} \hat{P}_1(x)  + n^{1/4}\hat{P}_2(x)   + \kO(1) \left( \hat{R}(x)n^{-1/4}  + p_1(x) + p_2(x) n^{-1/2} + r(x)n^{-1}  \right) }{n^{3/4} \hat{P}_1(-\infty)  + n^{1/4}\hat{P}_2(-\infty)   + \kO(n^{1/12})} \\
&=& \frac{ \hat{P}_1(x)  + n^{-1/2}\hat{P}_2(x)   + \kO(1) \left( \hat{R}(x)n^{-1}  + p_1(x) n^{-3/4} + p_2(x) n^{-5/4} + r(x)n^{-7/4}  \right) }{ \hat{P}_1(-\infty)  + n^{-1/2}\hat{P}_2(-\infty)   + \kO(n^{-2/3})}. 
\end{eqnarray*}
Notice that $1- F(x) = \hat{P}_1(x)/ \hat{P}_1(-\infty)$.  Therefore,
\begin{eqnarray*}
&& -1+ \frac{1-F_n(x)}{1-F(x)} = -1 + (1-F_n(x))\frac{\hat{P}_1(-\infty)}{\hat{P}_1(x)} \\
&& =-1 + \frac{ \hat{P}_1(-\infty)  +\frac{\hat{P}_1(-\infty)}{\sqrt{n}} \frac{\hat{P}_2(x)}{\hat{P}_1(x)}   + \kO(1) \left( n^{-1} \frac{\hat{R}(x)}{\hat{P}_1(x)}     + n^{-3/4} \frac{p_1(x)}{\hat{P}_1(x)}  + n^{-5/4} \frac{p_2(x)}{\hat{P}_1(x)}  + n^{-7/4} \frac{r(x)}{\hat{P}_1(x)}  \right) }{ \hat{P}_1(-\infty)  + n^{-1/2}\hat{P}_2(-\infty)   + \kO(n^{-2/3})}\\ 
&&= \frac{1}{\sqrt{n}} \left( \frac{\hat{P}_2(x)}{\hat{P}_1(x)} - \frac{\hat{P}_2(-\infty)}{\hat{P}_1(-\infty)} \right)+ \kO(1) \left( \frac{x^{12}}{n} + n^{-3/4} \right),
\end{eqnarray*}
where for the last line  we have used that
\begin{eqnarray*}
\frac{\hat{R}(x)}{\hat{P}_1(x)} =  \Theta( x^{12}), \hspace{0.6 cm}    \frac{p_1(x)}{\hat{P}_1(x)} = \Theta (1),  \hspace{0.6 cm} \frac{p_2(x)}{\hat{P}_1(x)} = \Theta (x^6), \hspace{0.6 cm}  \frac{r(x)}{\hat{P}_1(x)}  =\Theta ( x^{12}).
\end{eqnarray*}
In conclusion, 
\begin{equation*}
 \frac{1-F_n(x)}{1-F(x)} = 1+ \frac{G(x)}{\sqrt{n}} + \kO(1) \left( \frac{x^{12}}{n} + n^{-3/4} \right),
\end{equation*}
with $G(x)$ as in Theorem \ref{tcp}.
\subsubsection{Case $x \leq 10$}
Using \eqref{fbnx}, \eqref{ebnn} and the same arguments as in the case $x>10$, we can prove that 
\begin{eqnarray*}
&&  \frac{1-F_n(x)}{1-F(x)} = 1+ \frac{G(x)}{\sqrt{n}} + \kO(n^{-2/3}).
\end{eqnarray*}
Notice that here the term $ \kO(n^{-2/3})$ comes from the quotient $ \kO(n^{1/12})/ n^{3/4}$.

\begin{ack} \emph{
 We would like to thank the anonymous referee  for carefully reading the manuscript and many valuable comments.  }
 \end{ack}


\begin{thebibliography}{99}

\bibitem{COT} T. \c{C}a\v g\i n, P. E. Oliveiraa, N. Torradoa. {\it A moderate deviation for associated random variables}. J. Korean Statist. Soc. 45 (2016), no. 2, 285--294. 

\bibitem{CS} S. Chatterjee, Q. Shao. {\it Nonnormal approximation by Stein's method of exchangeable pairs with application to the Curie-Weiss model}. Ann. Appl. Probab. 21, 464--483 (2011). 

\bibitem{CFS} L. Chen, X. Fang, Q. Shao. {\it From Stein identities to moderate deviations}. Ann. Probab. 41, 262--293 (2013).

\bibitem{CSWX} X. Chen, Q.-M. Shao, W. B. Wu, L. Xu. {\it Self-normalized Cram\'er-type moderate deviations under dependence}. Ann. Statist. 44 (2016), no. 4, 1593--1617. 



\bibitem{EL} P. Eichelsbacher, M.  L\"owe. {\it  Stein's method for dependent random variables occurring in statistical mechanics}. Electron. J.  Probab.   15, no. 30, 962-988 (2010).


\bibitem{E} R. S. Ellis. {\it Entropy, large deviations, and statistical mechanics}. Grundlehren der Mathematischen Wissenschaften, 271. Springer-Verlag, New York, 1985.

\bibitem{EN}  R. S. Ellis, C. M. Newman. {\it Limit theorems for sums of dependent random variables occurring in statistical mechanics}. Z. Wahrsch. Verw. Gebiete, 117--139 (1978).



\bibitem{Grib} N. Gribkova. {\it Cram\'{e}r type moderate deviations for trimmed L-statistics}. Math. Methods Statist. 25 (2016), no. 4, 313--322.

\bibitem{LiSh1} W. Liu, Q.-M. Shao. {\it Cram\'{e}r-type moderate deviation for the maximum of the periodogram with application to simultaneous tests in gene expression time series}. Ann. Statist. 38 (2010), no. 3, 1913--1935.

\bibitem{LiSh} W. Liu, Q.-M. Shao. {\it A Carm\'{e}r moderate deviation theorem for Hotelling's $T^2$-statistic with applications to global tests}. Ann. Statist. 41 (2013), no. 1, 296--322.

\bibitem{PSZW} M. Peligrad, H. Sang, Y. Zhong, W. B. Wu. {\it Exact moderate and large deviations for linear processes}. Statist. Sinica 24 (2014), no. 2, 957--969.

\bibitem{P} V. V. Petrov.  {\it Sums of Independent Random Variables}. Springer, New York, 1975.

\bibitem{R} H. Robbins. {\it A Remark of Stirling's Formula}.  Amer. Math. Monthly 62, 26--29 (1955).

\bibitem{Shao} Q.-M. Shao, Q. Wang. {\it Self-normalized limit theorems: a survey}. Probab. Surv. 10 (2013), 69--93.


\bibitem{WuZh} W. B. Wu, Z. Zhao. {\it Moderate deviations for stationary processes}. Statist. Sinica 18 (2008), no. 2, 769--782.


\end{thebibliography}
\end{document}